\documentclass[11pt]{amsart}
\usepackage{geometry}                % See geometry.pdf to learn the layout options. There are lots.
\usepackage{graphicx}
\usepackage{amssymb,amscd}
\usepackage{epstopdf}
\usepackage{tikz}
\usepackage{hyperref}
\usetikzlibrary{arrows, shapes}
\newtheorem{thm}{Theorem}[section]
\newtheorem{lem}{Lemma}[section]
\newtheorem{cor}{Corollary}[section]
\newtheorem{pro}{Proposition}[section]

\newcommand{\nrmp}[1]{||#1 ||_{L^p}}

\newcommand{\agl}[1]{\langle#1\rangle}
\newcommand{\rn}[1]{\mathbb{R}^n}

\begin{document}
\title{On $L^p-$boundedness of pseudo-differential operators of Sj\"ostrand's class}
\author{Jayson Cunanan}
\address{Graduate School of Mathematics, Nagoya University, Furocho, Chikusa-ku, Nagoya 464-8602, Japan}
  \email{d12005n@math.nagoya-u.ac.jp}


\maketitle
\begin{abstract}
We extended the known result that symbols from modulation spaces $M^{\infty,1}(\mathbb{R}^{2n})$, also known as the Sj\"{o}strand's class, produce bounded operators in $L^2(\mathbb{R}^n)$, to general $L^p$ boundedness at the cost of lost of derivatives. Indeed, we showed that pseudo-differential operators acting from $L^p$-Sobolev spaces $L^p_s(\mathbb{R}^n)$ to $L^p(\mathbb{R}^n)$ spaces with symbols from the modulation space $M^{\infty,1}(\mathbb{R}^{2n})$ are bounded, whenever $s\geq n|1/p-1/2|.$ This estimate is sharp for all $1\leq p\leq\infty$.\end{abstract}

\section{Introduction}
Given a symbol $\sigma\in \mathcal{S}'(\mathbb{R}^{2n}),$ we define a pseudo-differential operator in $\mathbb{R}^n$ by
\begin{equation}
\sigma(x,D)f(x)=\int_{\mathbb{R}^n} e^{2\pi ix\cdot\xi}\sigma(x,\xi)\hat{f}(\xi)\ d\xi,\quad f\in \mathcal{S}(\mathbb{R}^{n}),
\end{equation}
where $\hat{f}(\xi)=\mathcal{F}f(\xi)=\int_{\mathbb{R}^n}e^{-i2\pi x\cdot\xi}f(x)\ dx$ is the Fourier transform of $f.$ It is clear that $\sigma(x,D)f$ is well-defined as a temperate distribution. These operators play essential roles in many fields including, partial differential equations (PDEs), quantum mechanics, and signal processing. Boundedness results of pseudo-differential operators on $L^p-$Sobolev spaces are of special interest since they imply regularity results of the corresponding PDEs. 

A thorough analysis of such operators have been carried on for H\"ormander's classes, $S^m_{\rho,\delta},m\in\mathbb{R},0\leq\delta\leq\rho\leq1,$ of smooth functions $\sigma(x,\xi)$ satisfying the estimates $$|\partial^{\alpha}_{x}\partial^{\beta}_{\xi}\sigma(x,\xi)|\leq C_{\alpha,\beta}(1+|\xi|)^{m-\rho|\beta|+\delta|\alpha|},$$ for all multi-indexes $\alpha$ and $\beta.$ In particular, the classical Calderon-Vaillancourt theorem \cite{CV71} states that symbols belonging to $S^0_{0,0}$ produce bounded operators on $L^2.$ In \cite{sj1,sj2}, Sj\"ostrand extended this result on $L^2$ boundedness by introducing a new class of symbols, containing $S^0_{0,0}$, which do not require boundedness of the derivatives of the symbols. This new class is later identified as the modulation space $M^{\infty,1}$, first introduced in time-frequency analysis by Feichtinger \cite{f1,f2,abl}. Roughly speaking, elements of $M^{p,q}$ are distributions with the same local regularity as a function whose Fourier transform is in $L^q$ and has decay properties of an $L^p$ function (see section 2 for details). In \cite{Gr06,GC99}, Gr\"ochenig and Heil then significantly extended Sj\"ostrand's result to boundedness of pseudo-differential operators on all modulation spaces.

For general $L^p$ boundedness, H\"ormander showed in \cite{hor67} that $m_p=n(1-\rho)|1/p-1/2|$ is the critical decreasing order for the $L^p$ boundedness of pseudo-differential operators in $S^m_{\rho,\delta}.$ Several sufficient conditions are known for the $L^p$ boundedness of pseudo-differential operators of H\"omander classes, we refer the reader to \cite{anv03} and the references therein for a detailed account. For symbols in $M^{\infty,1}$, it remains to know the minimal loss of derivatives required to achieve $L^p$ boundedness for all $1\leq p\leq\infty.$

Our goal in this note is to give sufficient and necessary conditions for the boundedness of pseudo-differential operators whose symbols come from $M^{\infty,1}(\mathbb{R}^{2n})$ acting on $L^p_s(\mathbb{R}^{n})$ to $L^p(\mathbb{R}^{n})$. The following theorem is our main result.

\begin{thm}
Let $1\leq p\leq\infty$. Then the pseudo-differential operator $\sigma: \mathcal{S}(\mathbb{R}^{n})\rightarrow\mathcal{S'}(\mathbb{R}^{n})$ having symbols in $M^{\infty,1}(\mathbb{R}^{2n})$ extends to a bounded operator from $L^p_s(\mathbb{R}^{n})$ to $L^p(\mathbb{R}^{n})$ if and only if $s\geq n|1/p-1/2|.$
\end{thm}

In classical literatures for $L^p$ boundedness of operators in H\"ormander's classes, results often require much stronger assumption, than simply belonging to a particular $S^m_{\rho,\delta},$ when dealing with the endpoints $p=1,\infty.$ As an illustration, it is known that the converse of Theorem 4.1 in Section 4 does not hold for $p=1$
 and $p=\infty.$ Although, Fefferman proved in \cite{fef73} the converse of Theorem 4.1 for the case $1<p<\infty.$ Here, Theorem 1.1 shows the advantage of taking symbols in $M^{\infty,1}(\mathbb{R}^{2n})$ since we are able to show boundedness for all $1\leq p\leq\infty.$

We organize this note as follows. Section 2 will consist of notations, definitions and properties of the involved functions spaces. In Section 3 and Section 4, we prove the sufficient and necessary conditions, respectively, for the boundedness of pseudo-differential operators $L^p_s(\mathbb{R}^{n})\rightarrow L^p(\mathbb{R}^{n}).$

\section{Preliminaries}

\textbf{Notations.} The Schwartz class of test functions on $\mathbb{R}^n$ shall be denoted by $\mathcal{S}(\mathbb{R}^n)$ and its dual, the space of tempered distributions, by $\mathcal{S}'(\mathbb{R}^n)$. The $L^p(\mathbb{R}^n)$ norm is given by $\nrmp{f}=(\int_{\mathbb{R}^n} |f(x)|^p\ dx)^{1/p} $ whenever $1\leq p<\infty,$ and $||f||_{L^{\infty}}=\text{ess.sup}_{x\in\mathbb{R}^n}|f(x)|$. The Fourier transform of a function $f\in \mathcal{S}(\mathbb{R}^n)$ is given by $$\mathcal{F}f(\xi)=\hat{f}(\xi)=\int_{\mathbb{R}^n}e^{-i2\pi x\cdot\xi}f(x)\ dx$$ which is an isomorphism of the Schwartz space $\mathcal{S}(\mathbb{R}^n)$ onto itself that extends to the tempered distributions $\mathcal{S'}(\mathbb{R}^n)$ by duality. The inverse Fourier transform is given by $\mathcal{F}^{-1}f(x)=\check{f}(x)=\int_{\mathbb{R}^n}e^{i2\pi \xi\cdot x}f(\xi)\ d\xi$. Given $1\leq p\leq\infty,$ we denote by $p'$ the conjugate exponent of $p$ (i.e. $1/p+1/p'=1$). The translation operator is defined by $T_xf(t)=f(t-x)$. We now recall the definitions of the function spaces to be used in this article.

\textbf{$L^p$-Sobolev:} We adapted Stein's notation \cite{ste} and define the $L^p$-Sobolev norm by
$$||f||_{L^p_s}=||((1+|\cdot|^2)^{s}\widehat{f}(\cdot))^\vee||_{L^p}.$$
We remark that this notation should not be interchanged with the set of all $f$ such that $(1+|x|^2)^{s/2}f$ belongs to $L^p.$

\textbf{Modulation spaces}. Let $ s\in \mathbb{R}, $ we denote $ \agl{\xi}^s=(1+|\xi|^2)^{s/2},\small \xi\in\mathbb{R}^n .$ The class of weight functions to be used in this note will be denoted by $\nu_{s_1,s_2}(x,\xi)=\agl{x}^{s_1}\agl{\xi}^{s_2},s_i\in\mathbb{R},i=1,2,$ on $\mathbb{R}^{2n}$ or $m(x,\xi,\zeta_1,\zeta_2)=\agl{x}^{s_1}\agl{\xi}^{s_2}=(\nu_{s_1,s_2}\otimes1)(x,\xi,\zeta_1,\zeta_2),$ on $\mathbb{R}^{4n}$. For $ 1\leq p,q\leq \infty,$ the modulation space $ M^{p,q}_m $ consists of all tempered distributions $f\in \mathcal{S}'(\mathbb{R}^n)$ such that the norm
$$||f||_{M^{p,q}_m}=\left(\int\limits_{\mathbb{R}^n}\left(\int\limits_{\mathbb{R}^n}|V_gf(x,\xi)|^pm(x,\xi)^p\ dx\right)^{q/p}d\xi\right)^{1/q},$$
is finite, with usual modifications if $p$ or $q= \infty$. Here $V_gf$ denotes the short-time Fourier transform (STFT) of $f\in\mathcal{S'}(\mathbb{R}^n)$ with respect to the window $0\neq g\in\mathcal{S}(\mathbb{R}^n)$ defined by $$V_gf(x,\xi)=\int_{\mathbb{R}^n}f(y)\overline{g(y-x)}e^{-2\pi iy\cdot\xi}\ dy.$$ The modulation space $M^{p,q}_m(\mathbb{R}^n)$ is a Banach space whose norm is independent of the choice of the window $g.$ If $m\equiv1$, we write $M^{p,q}$. A weighted modulation space in $\mathbb{R}^{2n}$ will be denoted by $M^{p,q}_{\nu_{s_1,s_2}\otimes1}(\mathbb{R}^{2n}).$

\textbf{Wiener amalgam spaces:} For $ 1\leq p,q\leq \infty,$ $ s\in\mathbb{R}$ and $0\neq g\in\mathcal{S}$, the Wiener amalgam space $ W^{p,q}_s(\mathbb{R}^n) $ consists of all tempered distributions $f\in \mathcal{S}'(\mathbb{R}^n)$ such that the norm
$$||f||_{W^{p,q}_s}:=||f||_{W({\mathcal F}L^q_s, L^p)}=\left(\int\limits_{\mathbb{R}^n}\left(\int\limits_{\mathbb{R}^n}|\mathcal{F}(f\cdot T_yg)(\omega)|^q\agl{\omega}^{sq}\ d\omega\right)^{p/q}dy\right)^{1/p},$$
is finite, with usual modifications if $p$ or $q= \infty$.
If $ s=0 $ we simply write $ W^{p,q} $ instead of $ W_0^{p,q} .$ We note that this definition is independent of the choice of window $g$. We denote the closure of the Schwartz class $\mathcal{S}(\mathbb{R}^n)$ in the $W^{p,q}_s$-norm by $\mathcal{W}^{p,q}_s(\mathbb{R}^n).$ If $1\leq p,q<\infty$ then $\mathcal{W}^{p,q}_s=W^{p,q}_s.$

We collect properties of Wiener amalgam spaces in the following lemma. We note also that analogous properties hold for modulation spaces. In fact, $\mathcal{F}W^{p,q}=M^{p,q}.$
\begin{lem}
Let $p,q,p_i.q_i\in[1,\infty]$ for $i=1,2$ and $s_j\in\mathbb{R}$ for $\ j=1,2.$ Then

\begin{enumerate}
\item $\mathcal{S}(\mathbb{R}^n)\hookrightarrow W^{p,q}(\mathbb{R}^n)\hookrightarrow \mathcal{S'}(\mathbb{R}^n);$
\item $\mathcal{S}$ is dense in $W^{p,q}$ if $p$ and $q<\infty;$
\item If $q_1\leq q_2$ and $p_1\leq p_2$, then $W^{p_1,q_1}\hookrightarrow W^{p_2,q_2};$
\item If $s_1\geq s_2$, then $W^{p,q}_{s_1}\hookrightarrow W^{p,q}_{s_2};$
\item $\agl{D}^{-s}:W^{p,q}\rightarrow W^{p,q}_{s}$, $f \mapsto (\widehat{f}(\cdot)
\langle \cdot \rangle^{-s})^\vee$ , is an isomorphism;
\item (Convolution) If $\mathcal{F}L^{q_1}\ast\mathcal{F}L^{q_2}\hookrightarrow\mathcal{F}L^{q}$ and $L^{p_1}\ast L^{p_2}\hookrightarrow L^{p},$ then $$W(\mathcal{F}L^{q_1},L^{p_1})\ast W(\mathcal{F}L^{q_2},L^{p_2})\hookrightarrow W(\mathcal{F}L^{q},L^p).$$ 
\item (Complex interpolation) For $0<\theta<1.$ Let $\dfrac{1}{p}=\dfrac{\theta}{p_1}+\dfrac{1-\theta}{p_2}$, $\dfrac{1}{q}=\dfrac{\theta}{q_1}+\dfrac{1-\theta}{q_2}$ and $s=\theta s_1+ (1-\theta)s_2.$ Then
$$[\mathcal{W}^{p_1,q_1}_{s_1}, \mathcal{W}^{p_2,q_2}_{s_2}]_{[\theta]}=\mathcal{W}^{p,q}_s;$$
\item (Duality) $(\mathcal{W}^{p,q}_s)'= W^{p',q'}_{-s},$ where $1/p+1/p'=1=1/q+1/q',\quad p,q\neq\infty.$
\end{enumerate}
\end{lem}
The proofs of these statements can be found in \cite{f1,f3,abl,tof2}. 

Pseudodifferential operators in the form (1) are referred as Kohn-Nirenberg correspondence. Now we give another form which is called the Weyl quantization, and can be defined as follows:
$$L_{\sigma}f(x)=\int\int e^{2\pi (x-y)\xi}\sigma(\dfrac{x+y}{2},\xi)f(y)\ dyd\xi.$$
Our interest in using this form is due to its adjoint, namely, $L^{*}_{\sigma}=L_{\bar{\sigma}}$. Moreover, for symbols in modulation spaces, boundedness results can be obtained using either the Kohn-Nirenberg or Weyl form. This statement will be made clear by Remark 2.1.

We quote the following result in \cite[Theorem 14.3.5]{grtfa}.

\begin{thm}
Let $T$ be a continuous linear operator mapping $\mathcal{S}(\mathbb{R}^n)$ into $\mathcal{S}'(\mathbb{R}^n).$ Then there exist tempered distributions $K,\sigma,a\in\mathcal{S}'(\mathbb{R}^{2n})$, such that $T$ has the following representations:
\begin{enumerate}
\item[(i)] as an integral operator $\langle Tf,g\rangle=\langle K,g\otimes\bar{f}\rangle,$ for $f,g\in\mathcal{S}(\mathbb{R}^n)$;
\item[(ii)] as a pseudodifferential operator $T=L_{\sigma},$ with Weyl symbol $\sigma$ and $T=a(x,D)$ with Kohn-Nirenberg symbol $a$.
\end{enumerate}
The relations between $K,\sigma,a$ are given by
\begin{equation}
\sigma=\mathcal{F}_2\tau_sK,\quad a=\mathcal{U}\sigma
\end{equation}
where $\mathcal{F}_2$ is the partial Fourier transform in the second variable, $\tau_s$ is the symmetric coordinate transformation $\tau_sK(x,y)=K(x+y/2,x-y/2)$ and the operator $\mathcal{U}$ is defined by $\widehat{\mathcal{U}\sigma}(\omega,u)=e^{\pi i\omega u}\hat{\sigma}(\omega,u).$
\end{thm}
\textbf{Remark 2.1.}
Since, $\sigma(x,D)=L_{\mathcal{U}^{-1}\sigma},$ and the fact that $M^{\infty,1}$ is invariant under $\mathcal{U}^{-1}$ implies that boundedness results for Kohn-Nirenberg and Weyl form are interchangeable for symbols in $M^{\infty,1}$ (see \cite[Corollary 14.5.5]{grtfa}).

\section{Sufficient condition}
The proof for the sufficient conditions of Theorem 1.1 utilize the inclusion relations between $L^p-$Sobolev and Wiener amalgam spaces $W^{p,q}$ described in \cite{cun2}. For $(1/p,1/q)\in [0,1]\times [0,1]$, the indices $\tau_1(p,q)$ and $\tau_2(p,q)$ are defined as follows:

\begin{equation*}
\tau_1(p,q) =
\begin{cases}
0 & \text{if $(1/p,1/q)\in I^*_1:$ min$(1/p',1/2)\geq 1/q$}\\
1/p+1/q-1 & \text{if $(1/p,1/q)\in I^*_2:$ min$(1/q,1/2)\geq 1/p'$}\\
1/q-1/2 & \text{if $(1/p,1/q)\in I^*_3:$ min$(1/p',1/q)\geq 1/2$}
\end{cases}
\end{equation*}

\begin{equation*}
\tau_2(p,q) =
\begin{cases}
0 & \text{if $(1/p,1/q)\in I_1:$ max$(1/p',1/2)\leq 1/q$}\\
1/p+1/q-1 & \text{if $(1/p,1/q)\in I_2:$ max$(1/q,1/2)\leq 1/p'$}\\
1/q-1/2 & \text{if $(1/p,1/q)\in I_3:$ max$(1/p',1/q)\leq 1/2$}
\end{cases}
\end{equation*}

where $1/p+1/p'=1=1/q+1/q'.$ Refer to Figure 1 for a visualization.

\begin{thm}[\cite{cun2}]
Let $1\leq p,q\leq\infty$ and $s\in\mathbb{R}.$ Then $L_s^p\hookrightarrow W^{p,q}$ if one of the following conditions is satisfied.
\begin{enumerate}
\item $p>q, q<2$ and $s> n\tau_1(p,q);$
\item $p\neq 1,$ max$(1/p,1/2)\geq1/q$ and $s\geq n\tau_1(p,q);$
\item $p=1, q=\infty$ and $s\geq n\tau_1(1,\infty);$
\item $p=1, q\neq\infty$ and $s> n\tau_1(1,q).$
\end{enumerate}
%Conversely, if $L_s^p\hookrightarrow W^{p,q}$, then $s\geq n\tau_1(p,q).$

\end{thm}

\begin{thm}[\cite{cun2}]
Let $1\leq p,q\leq\infty$ and $s\in\mathbb{R}.$ Then $W^{p,q}\hookrightarrow L_s^p$ if one of the following conditions is satisfied.
\begin{enumerate}
\item $p<q, q>2$ and $s< n\tau_2(p,q);$
\item $p\neq \infty,$ min$(1/p,1/2)\leq1/q$ and $s\leq n\tau_2(p,q);$
\item $p=\infty, q=1$ and $s\leq n\tau_2(\infty,1);$
\item $p=\infty, q\neq1$ and $s< n\tau_2(1,q).$
\end{enumerate}
%Conversely, if $W^{p,q}\hookrightarrow L_s^p$, then $s\leq n\tau_2(p,q).$

\end{thm}

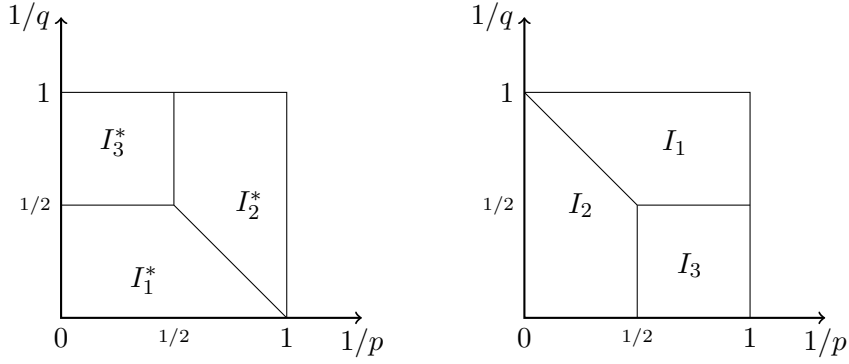
\begin{figure}
\begin{tikzpicture}[scale=2]
    % Draw axes
    \draw [<->,thick] (0,2) node (yaxis) [left] {$1/q$}
        |- (2,0) node (xaxis) [below] {$1/p$};
\path[draw] (0,1.5) node[left] {1} -- (1.5,1.5) -- (1.5,0) node[below] {$1$};
\path[draw] (1.5,0)  -- (1.5/2,1.5/2) --(1.5/2,1.5);
\path[draw] (1.5/2,1.5/2) -- (0,1.5/2) ;
%\path[draw] (0,0) -- (0,1.5/2) node[left]  {\tiny 1/2};
\node [left]at (0,1.5/2) {\tiny 1/2};
\node[below] at (1.5/2,0) {\tiny 1/2};
\node[below] at (0,0) {0};
\node at (3/4+.5,1.5/2) {$I^{\ast}_2$};
\node at (3/4-.2,1/4) {$I^{\ast}_1$};
\node at (3/4-.4,1+1/6) {$I^{\ast}_3$};
\end{tikzpicture}
\quad
\quad
\begin{tikzpicture}[scale=2]
    % Draw axes
    \draw [<->,thick] (0,2) node (yaxis) [left] {$1/q$}
        |- (2,0) node (xaxis) [below] {$1/p$};
\path[draw] (0,1.5) node[left] {1} -- (1.5,1.5) -- (1.5,0) node[below] {$1$};
\path[draw] (1.5/2,0)  -- (1.5/2,1.5/2);
\path[draw] (1.5/2,1.5/2) -- (1.5,1.5/2)  ;
\path[draw] (1.5/2,1.5/2) -- (0,1.5)  ;
%\path[draw] (0,0) -- (0,1.5/2) node[left]  {\tiny 1/2};
\node [left]at (0,1.5/2) {\tiny 1/2};
\node[below] at (0,0) {0};
\node at (1.5/4,1.5/2) {$I_2$};
\node at (1+.1,1/4+.1) {$I_3$};
\node at (1,1+1/6) {$I_1$};
\node[below] at (1.5/2,0) {\tiny 1/2};
\end{tikzpicture}
\caption{The index sets for  $L^p_s-W^{p,q}$ inclusion}
\label{fig1}
\end{figure}

In the work of Cordero, Tabacco and Wahlberg \cite[Corollary 3.11]{ctw13}, they showed how symbols in modulation spaces $M^{p,q}(\mathbb{R}^{2n})$ produces bounded Fourier integral operators in Wiener amalgam spaces $W^{p,q}(\mathbb{R}^{n})$. Choosing the phase $\Phi(x,\xi)=x\cdot\xi$, this result reduces to boundedness of pseudo-differential operators of Kohn-Nirenberg correspondence. For the ease of reference, we write a particular case of their result as follows. 

\begin{cor}
Given $\sigma\in M^{\infty,1}(\mathbb{R}^{2n}).$ Then the operator $\sigma(x,D)$ extends to a bounded operator on Wiener amalgam spaces $W^{p,q}(\mathbb{R}^{n}).$
\end{cor}

\begin{proof}[Proof of IF part on Theorem 1.1]
First, we let
\begin{equation*}
q=2\quad \text{if } 1\leq p\leq2.
\end{equation*}
Then using the embedding results Theorem 3.1 and Theorem 3.2, we can say that $W^{p,q}\hookrightarrow L^p$ and $L^p_s\hookrightarrow W^{p,q}$ whenever $s\geq n(1/p+1/q-1)$ for $1\leq p\leq2$. The desired boundedness now follows from the following commutative diagram.

\begin{equation*}
\begin{CD}
W^{p,q} @>\sigma>> W^{p,q}\\
@AInclusionAA @VVInclusionV\\
L^p_s @>\sigma>> L^p
\end{CD}
\end{equation*}

Substituting the appropriate value of $q$ we have $s\geq n(1/p-1/2)$ for $1\leq p\leq2$. By Remark 2.1 the result for $p>2$ follows by duality on the corresponding Weyl quantization.

\end{proof}

\textbf{Remark 3.1.} The inclusion relations between $L^p-$Sobolev spaces and modulation spaces $M^{p,q}$ are known due to \cite{kob}. Using a similar argument as above, one can conclude boundedness results but with worse estimates for $s$. Namely, $s\geq 2n|1/p-1/2|$.

\section{Necessary condition}
%\begin{lem}\label{tes}
%Let $h\in\mathcal{C}_0^{\infty}$, and consider the family of functions
%$$h_{\lambda}(x)=h(x)e^{-i\pi\lambda|x|^2},\quad \lambda\geq1.$$
%Then, $\hat{h}(\xi)=O(\lambda^{-n/2})$
%\end{lem}

Boundedness results of pseudo-differential operators having symbols in weighted modulation spaces acting on $L^p$ spaces is equivalent to boundedness results of operators with symbols in unweighted modulation spaces acting on $L^p-$Sobolev spaces. More specifically, we have the following proposition.

\begin{pro}
Let $\sigma(x,D)$ and $\tilde{\sigma}(x,D)$ be pseudodifferential operators such that
$$\tilde{\sigma}(x,\xi)=\sigma(x,\xi)\agl{\xi}^s,\quad x,\xi\in \mathbb{R}^n,s\in \mathbb{R}.$$
Then
\begin{enumerate}
\item the operator $\sigma(x,D)$ is bounded from $L^p(\mathbb{R}^n)$ to itself if and only if $\tilde{\sigma}(x,D)$ is bounded from $L^p_s(\mathbb{R}^n)$ to $L^p(\mathbb{R}^n)$.
\item \begin{equation*}
\sigma\in M^{\infty,1}_{\nu_{0,s}\otimes 1}(\mathbb{R}^{2n})\Longleftrightarrow \tilde{\sigma}\in M^{\infty,1}(\mathbb{R}^{2n}).
\end{equation*}
\end{enumerate}
\end{pro}
\begin{proof}
The proof of (1) is a direct consequence of the following commutative diagram. Indeed, the vertical arrows define isomorphisms:
\begin{equation*}
\begin{CD}
L^p @>\sigma>> L^p\\
@A\agl{D}^sAA @VVIdentityV\\
L^p_s @>\tilde{\sigma}>> L^p
\end{CD}
\end{equation*}
Also, by \cite[Corollary 2.3]{tof2}, multiplication by a weight function $\eta(x,\xi,\zeta_1,\zeta_2)=\agl{\xi}^s $, $x,\xi,\zeta_1,\zeta_2\in \mathbb{R}^n$ is an isomorphism from $M^{\infty,1}_{\nu_{0,s}\otimes 1}(\mathbb{R}^{2n})$ to $M^{\infty,1}(\mathbb{R}^{2n}),$ thus $(2)$.
\end{proof}
\textbf{Remark.} For the general case of Fourier integral operators of Proposition 4.1, we refer the reader to \cite[Proposition 2.8]{cor}, where the authors rephrased boundedness of FIO with symbols in weighted modulation spaces acting on unweighted modulation spaces by FIO with symbols in unweighted modulation spaces acting on weighted modulation spaces.

We recall the following theorem by H\"ormander which gives the critical decreasing order for the $L^p$ boundedness of pseudo-differential operators $\sigma(x,D)$ in $\textnormal{Op} (S^s_{\rho,\delta}).$
\begin{thm}[H\"ormander]

Let $0\leq \delta\leq \rho\leq 1$ and $\delta<1$. Then
$$\textnormal{Op} (S^s_{\rho,\delta})\subset \mathcal{L}(L^p(\mathbb{R}^n)) \Longrightarrow s\leq -n(1-\rho)|1/p-1/2|$$
\end{thm}

\begin{proof}[Proof of the converse of Theorem 1.1]
By Proposition 4.1 it suffice to show that 
\begin{equation}\label{op}
\textnormal{Op} (M^{\infty,1}_{\nu_{0,s}\otimes 1}(\mathbb{R}^{2n}))\subset \mathcal{L}(L^p(\mathbb{R}^n)) \Longrightarrow s\geq n|1/p-1/2|.
\end{equation}

In \cite{GC99}, Gr\"ochenig and Heil proved the following embedding $S^0_{0,0}\hookrightarrow M^{\infty,1}$. Using the fact that,$\agl{D}^{-s}M^{\infty,1}(\mathbb{R}^{2n})=M^{\infty,1}_{\nu_{0,s}\otimes 1}(\mathbb{R}^{2n}),\sigma(x,D)\agl{D}^{-s}\in \textnormal{Op} (S^{-s}_{0,0})$ whenever $\sigma\in S^{0}_{0,0},$ we conclude that $S^{-s}_{0,0}\hookrightarrow M^{\infty,1}_{\nu_{0,s}\otimes 1}(\mathbb{R}^{2n}).$ Using this inclusion on (\ref{op}) together with Theorem 4.1 gives the desired estimate.
\end{proof}

\bibliographystyle{abbrv}
\bibliography{research}

\end{document}